\documentclass[reqno, letterpaper,11pt]{amsart}
\usepackage{amsmath,enumerate,enumitem,amsfonts,color,amsthm,hyperref,bbm} 
\usepackage{appendix}
\title[Scaling Limits for Exponential Hedging]{Scaling Limits for Exponential Hedging in the Brownian Framework}

\author{Yan Dolinksy} 
\address{Department of Statistics, Hebrew University}
\email{yan.dolinsky@mail.huji.ac.il}

\author{Xin Zhang} 
\address{Department of Finance and Risk Engineering, New York University}
\email{xz1662@nyu.edu}

\date{\today}

\numberwithin{equation}{section}  
\newtheorem{defn}{Definition}[section]

\newtheorem{rem}[defn]{Remark}
\newtheorem{thm}[defn]{Theorem}
\newtheorem{prop}[defn]{Proposition}

\newtheorem{lem}[defn]{Lemma}
\newtheorem{asm}[defn]{Assumption}

\DeclareMathOperator{\Tr}{Tr}

\newcommand{\R}{\mathbb R}

\begin{document}

\begin{abstract}
In this paper, we consider scaling limits of exponential utility indifference prices for European contingent claims in the Bachelier model.
We show that the scaling limit can be represented in terms of the \emph{specific relative entropy}, and in addition we construct asymptotically optimal hedging strategies. To prove the upper bound for the limit, we formulate the dual problem as a stochastic control, and show there exists a classical solution to its 
{Hamilton-Jacobi-Bellman} (HJB) equation. The proof for the lower bound relies on the duality result for exponential hedging in discrete time. 

\end{abstract}

\thanks{Y. Dolinsky is partially supported by the ISF grant 305/25. X. Zhang is partially supported by the NSF grant DMS-2508556.}

\maketitle

\section{Introduction}

Taking into account market frictions is an important challenge in financial modeling.
In this paper, we focus on the friction that the
rebalancing of the portfolio strategy is limited to occur discretely.
In such a realistic
situation, a general future payoff cannot be hedged perfectly even in complete market models
such as the Bachelier model or the Black--Scholes model. 

In the present study, 
we follow the well known approach of utility indifference pricing which is commonly used in
the setup of incomplete markets (see \cite{C} and the references therein).
Generally speaking, this approach says that the
price of a given contingent claim should be equal to the minimal amount of money that an investor has to
be offered so that she becomes indifferent (in terms of utility) between the situation where she has sold the
claim and the one where she has not.

More precisely, we consider the multi-dimensional Bachelier model, 
for the setup where the investor can hedge on
an equidistant set of times.
For a given vanilla European contingent claim, under some regularity assumptions, 
we provide the asymptotic behavior of the exponential utility indifference prices
where the risk aversion goes to infinity linearly
in the number of trading times.
In Theorem~\ref{thm.1}, we show that the limit of the exponential utility indifference prices can be represented in terms of the so-called {\em{specific relative entropy}} between martingales and Brownian motion. As the major component of the argument, in Theorem~\ref{thm.2} we construct asymptotically optimal hedging strategies for the corresponding utility maximization problem.  

Introduced by Gantert \cite{gantert1991einige}, as a scaling limit of relative entropy, the specific relative entropy serves as the rate function of a large deviation principle of Brownian paths. Its explicit formula has been derived in \cite{BB24,BaUn23,BCH24}. Moreover, it has found applications in stochastic optimization \cite{BB,BWZ24,Guo23} and volatility surface calibration \cite{Av01}. It is well known (see, for instance, Chapter 3 in \cite{F}), that  relative entropy appears in the dual representation of the exponential utility maximization problem. Therefore the specific relative entropy naturally arises in the asymptotic limit of exponential utility indifference prices. This has been observed in \cite{CD} for the one dimensional case. 

The present work is a natural continuation of \cite{CD}.
While the paper \cite{CD} considered European contingent claims with path dependent payoffs, the proof of the upper bound was not 
complete. The main tool in \cite{CD} was duality theory.
Hence, due to technical reasons, in the proof of the upper bound, to guarantee the compactness of the pre-limit, the paper
imposes an artificial term in discrete-time hedging. The goal of this paper
is to remove this extra term in the limit theorem and construct optimal
hedging strategies in the case of vanilla options. Let us emphasize that \cite{CD} does not provide any results for optimal trading strategies. 
 
In order to prove the upper bound of the limit, we formulate the dual problem as a continuous-time stochastic control. We prove that under some assumptions on the terminal payoff function, the HJB equation of this control problem is uniformly elliptic, and thus has a classical solution. It is based on the observation that the value function of control problem inherits the concavity of the terminal payoff, and thus the optimal control is uniformly bounded away from zero and from above. Using the regularity of the value function, we construct an asymptotic optimal strategy for the primal hedging problem in the same flavor as \cite{expert_4,pde_approach}, and hence obtain the upper bound.  To show the lower bound, we first prove by a density argument that it is sufficient to restrict the optimal control of the dual problem to be piecewise deterministic controls. Then we conclude the result by invoking the discrete-time duality theorem. This is the multidimensional extension of the lower bound from \cite{CD}.

The rest of the paper is organized as follows. In Section~\ref{sec:2}, we introduce the model and the main results Theorem~\ref{thm.1} and Theorem~\ref{thm.2}.
In Section ~\ref{sec:3} we prove Theorem~\ref{thm.2}. In Section ~\ref{sec:4}
we prove Theorem~\ref{thm.1}.

\section{Preliminaries and Main Results}\label{sec:2}
Let $T=1$ be the time horizon and let
$W=\left(W^1_t,...,W^d_t\right)_{t \in [0,1]}$
be a standard $d$-dimensional Brownian motion defined on a filtered probability space
$(\Omega, \mathcal{F},(\mathcal F_t)_{0\leq t\leq 1}, \mathbb P)$
where the filtration $(\mathcal F_t)_{0\leq t\leq 1}$ satisfies the
usual conditions (right continuity and completeness).
We consider a simple financial market
with a riskless savings account bearing zero interest (for simplicity) and with $d$-risky
asset $S=\left(S^1_t,...,S^d_t\right)_{t \in [0,1]}$ with Bachelier price dynamics
\begin{equation*}\label{model}
S_t=S_0+b t+ W_t, \ \ t\in [0,1] 
\end{equation*}
where $S_0\in \mathbb R^d$ is the initial position of the risky assets and 
$b \in \mathbb R^d$ is the constant drift. 

Fix $n\in\mathbb N$ and consider an investor who can trade the risky asset
only at times from the grid $\left\{0,1/n,2/n,...,1\right\}$.
Namely, the set $\mathcal {A}^n$ of trading strategies for the $n$-step model is the set of all processes 
$\gamma=\{\gamma_k=(\gamma^1_k,...,\gamma^d_k)\}_{0\leq k\leq n-1}$
such that for any $k$, $\gamma_k$ are $\mathcal F_{k/n}$-measurable.
The corresponding portfolio value at the maturity date is 
\begin{equation*}\label{portfolio}
V^{\gamma}_{1}:=\sum_{i=0}^{n-1}\left\langle \gamma_i ,S_{(i+1)/n}-S_{i/n}\right\rangle.
\end{equation*}
As usual,
$\langle \cdot,\cdot\rangle$ 
denotes the
standard scalar product in Euclidean space $\R^d$.

Let $f:\mathbb R^d\rightarrow\mathbb R$ and consider a (vanilla) European contingent claim with the payoff
$f(S_1)$. The investor will assess the quality of a hedge by the resulting expected utility. 
For the $n$-step model we assume 
exponential utility with constant absolute risk aversion
equal to $n$. 

The corresponding \textit{certainty equivalent} 
does not depend on the investor's initial wealth
and is given by (for details see Chapter 2 in \cite{C})
\begin{equation*}\label{2.2}
c_n:=\frac{1}{n}\log\left(\inf_{\gamma\in\mathcal A^n}\mathbb E_{\mathbb P}\left[\exp\left(n\left(f(S_1)-V^{\gamma}_1\right)\right)\right]\right)
\end{equation*}
where $\mathbb E_{\mathbb P}$ denotes the expectation with respect to $\mathbb P$.
{ Economically speaking, the term $c_n$ is the amount of certain wealth that makes the investor to be indifferent between 
the following: (i) Selling the option and hedging it. (ii) Doing nothing. Namely, $c_n$ is the unique 
solution of the equation 
$$\inf_{\gamma\in\mathcal A^n}\mathbb E_{\mathbb P}\left[\exp\left(n\left(f(S_1)-V^{\gamma}_1-c_n\right)\right)\right]=\exp(n0)=1.$$}
\begin{rem}
Let us notice that from the scaling of property of Brownian motion, there is no loss of generality assuming 
that the maturity date is $T=1$. In addition, by modifying the payoff function $f$ and taking a linear transformation 
of the stock price vector, there is no loss of generality to assume that the (constant) volatility matrix 
of the Bachelier model is the identity matrix. Moreover, by modifying the payoff function $f$ we can 
generalize our setup to the case where the risk aversion in the $n$-step model is equal to $\ell n$ for some 
constant $\ell>0$.
\end{rem}

The first main result 
provides the scaling limit for the \textit{certainty equivalent} $c_n$, $n\in\mathbb N$,
under the following assumptions on $f$.  Let $\mathcal M_d$ be the set of all $d\times d$ real matrices, let $\mathcal S_d\subset \mathcal M_d$ be the set of all positive semi-definite matrices, and $I_d$ denote the identity matrix. We denote by $||\cdot||$
the spectral norm  on the space $\mathcal M_d$. For any $K>0 $,  denote $B_K:=\{x \in \R^d: \, |x| \leq K \}$ and  $B_K^c:= \{ x \in \R^d: \, |x| >K\}$
(as usual $|\cdot|$ denotes the Euclidean norm on $\mathbb R^d$). 
\begin{asm}\label{asm1}
The payoff function $f:\mathbb R^d\rightarrow\mathbb R$ is uniformly bounded and lower semi-continuous. There exists $M>0$ such that 
$f\in C^2(B^c_M;\R)$, and the corresponding gradient $\nabla f: B^c_M\rightarrow \mathbb R^d$ is uniformly bounded. Moreover, there exists a constant $\alpha>0$ such that $\nabla^2 f (x) \leq (1-\alpha) I_d$, $\forall \, x \in B_M^c$ (as usual we use the Loewner order for matrices). 
\end{asm}

\begin{thm}\label{thm.1}
Under Assumption \ref{asm1}, the \textit{certainty equivalent} of $f(S_1)$ has the scaling limit
\begin{equation}\label{2.3}
 \lim_{n\rightarrow\infty} c_n=
 \sup_{\Sigma \in\mathcal V}\mathbb E_{\mathbb P}\left[f\left(S_0+\int_{0}^1 \sqrt{\Sigma_u} \, dW_u\right)- \int_{0}^1 G\left(\Sigma_t\right)dt \right],
\end{equation}
where $\mathcal V$ denotes the class of all 
$\mathcal(\mathcal F_t)$-predictable processes $\Sigma=\{\Sigma_t\}_{t \in [0,1]}$  valued in $\mathcal S_d$ 
such that $\mathbb E_{\mathbb P}\left[\int_{0}^1 \Tr (\Sigma_t) \, dt\right]<\infty$ and
$$G(\Sigma):=\frac{1}{2}\left(\Tr(\Sigma)-d-\log\left(\det(\Sigma)\right)\right), \ \  \forall\, \Sigma\in\mathcal S_d.$$
\end{thm}

\begin{rem}
The utility indifference price in the $n$-step model is given by (see \cite{C}) 
$$\pi_n:=\frac{1}{n}\log\left(\frac{\inf_{\gamma\in\mathcal A^n}\mathbb E_{\mathbb P}\left[\exp\left(n\left(f(S_1)-V^{\gamma}_1\right)\right)\right]}
{\inf_{\gamma\in\mathcal A^n}\mathbb E_{\mathbb P}\left[\exp\left(-nV^{\gamma}_1\right)\right]}\right).$$
Using similar arguments as in \cite{CD} { (see the details below formula (3.7) there)}
it follows that 
$$\lim_{n\rightarrow\infty}\frac{1}{n}\log\left(
\inf_{\gamma\in\mathcal A^n}\mathbb E_{\mathbb P}\left[\exp\left(-nV^{\gamma}_1\right)\right]\right)=0.$$ 
Hence, $\lim_{n\rightarrow\infty}\pi_n=\lim_{n\rightarrow\infty} c_n$ and we can 
focus on the limiting behavior of the 
the \textit{certainty equivalent}.
\end{rem}
\begin{rem}
Let $\Sigma\in\mathcal S_d$ and let $\lambda_1,...,\lambda_d \geq 0$ be the corresponding eigenvalues of $\Sigma$.
Then,
$$G(\Sigma)=\frac{1}{2}\sum_{i=1}^d (\lambda_i-\log\lambda_i-1)\geq 0$$
where the last equality follows from the fact that the function $\mathbb R_+ \ni x\rightarrow x-\log x-1$ is a nonnegative function with the minimum point $x=1$. 
Moreover, since $f$ is uniformly bounded, then for any $\Sigma \in\mathcal V$ the right hand side of (\ref{2.3}) is well defined.
\end{rem}
\begin{rem}
    Under some regularity assumptions on $\Sigma$, it can be shown that the term $\mathbb E_{\mathbb P}\left[ \int_0^1 G(\Sigma_t) \,dt \right]$ is exactly the specific relative entropy between the martingale $M_t:=\int_0^t \sqrt{\Sigma_s} dW_s$ and the standard Brownian motion $W$. See \cite{BB24} and\cite{BaUn23}  for more details. 
\end{rem}

Next, we define the following regularity class of functions.
\begin{defn}
Let 
$\mathcal{H}$ be the set of all functions $h\in C^{\infty}(\mathbb R^d;\R)$ 
such that the derivatives of all order (including order zero) are uniformly bounded.
Moreover, there exists a constant $\alpha>0$ such that $\nabla^2 h(x) \leq (1-\alpha)I_d$, $\forall \, x \in \mathbb R^d$.
\end{defn}

We arrive at the following result 
which provides asymptotically optimal hedging strategies via particular PDEs with regular terminal conditions. 

\begin{thm}\label{thm.2}
Under Assumption~\ref{asm1}, for any $\epsilon>0$
 there exists $h\in \mathcal{H}$ such that
\begin{equation}\label{density}
 \sup_{x\in\mathbb R^d}\left|h(x)-x^2/2-\left(f- | \cdot |^2/2 \right)^{cav}(x)\right|<\epsilon,
 \end{equation}
where $\left(f- |\cdot |^2/2 \right)^{cav}$ denotes the concave envelop of 
$ x\mapsto f(x)-\frac{1}{2}|x|^2$.

For such $h$, the Cauchy problem
\begin{equation}\label{cauchy}
\partial_t u=\frac{1}{2}\log\left(\det(I_d-\nabla^2 u)\right), \ \ u(1,x)=h(x)
\end{equation}
has a unique classical solution $u(t,x)\in C^{1,2}([0,1]\times\mathbb R^d)$ and
we have 
\begin{align}\label{optimap}
&\lim\sup_{n\rightarrow\infty}\frac{1}{n}
\log\mathbb E\left[\exp\left(n\left(f(S_1)-\sum_{i=0}^{n-1} \left\langle \nabla u\left(\frac{i+1}{n},S_{\frac{i}{n}}\right),
S_{\frac{i+1}{n}}-S_{\frac{i }{n}}\right\rangle\right)\right)\right]\nonumber\\
&\leq 2\epsilon+\sup_{\Sigma\in\mathcal V}\mathbb E_{\mathbb P}\left[f\left(S_0+\int_{0}^1 \sqrt\Sigma_u \, dW_u\right)- \int_{0}^1 G\left(\Sigma_t \right)dt \right].
\end{align}
Namely, 
the constructed hedging strategies $\gamma^n:=\left(\nabla u\left(\frac{i+1}{n},S_{\frac{i}{n}}\right)\right)_{0\leq i\leq n-1}$, $n\in\mathbb N$,
are asymptotically $2\epsilon$-optimal. 
\end{thm}

We end this section with the following four remarks.
{\begin{rem}
The above type of scaling limits goes back to the seminal work of Barles and Soner \cite{BS}, which determines the scaling limit of utility indifference prices of vanilla options for small proportional transaction costs and high risk aversion.
In general, the common ground between the work of Barles and Soner and the present paper is that both 
 start with complete markets and consider small frictions, which make the markets incomplete and the derivative securities cannot be perfectly hedged with a reasonable initial capital. Then, instead of considering perfect hedging, these papers study utility indifference prices with exponential utilities and with large risk aversion.
 We show that if the risk aversion goes to infinity linearly
in the number of trading times we get a meaningful scaling limit. If the risk aversion goes to infinity faster than linearly, the indifference prices will converge 
to the model free price which is equal to $f^{cav}(S_0)$ ($f^{cav}$ denotes the concave envelop of $f$). If the risk aversion goes to infinity slower than linearly 
the indifference prices will converge to $\mathbb E_{\mathbb P}\left [f(S_0+W_1)\right]$ which is the unique price of the continuous time complete 
Bachelier model (for more details see Remark 2.3 in \cite{CD}).
\end{rem}}
\begin{rem}
As we will see in Section \ref{sec:3}, Assumption \ref{asm1} will be essential for the proof of the upper bound and
the construction of asymptotically optimal hedging strategies in Theorem~\ref{thm.2}. 
From mathematical finance point of view, the Bachelier model can be applied for the case where the
maturity date is small. Since, we scale the maturity date to $T=1$, it means that the components of the initial stock price should be 
large positive numbers. In this case, with large probability the stock price process remains non-negative.

Thus, two major type of payoffs which satisfy our assumption are put options given by
$$f(S_1):=\left(K-\sum_{i=1}^d a_i |S^i_1|\right)^{+}, \ \ K,a_1,...,a_d\geq 0$$
and truncated call options given by 
$$f(S_1):=\min\left(K_1,\left(\sum_{i=1}^d a_i |S^i_1|-K_2\right)^{+}\right), \ \ K_1,K_2,a_1,...,a_d\geq 0.$$
Moreover, let us observe that if Theorems \ref{thm.1}, \ref{thm.2}
hold true for $f(S_1)$ 
then they also hold true for $c_0+\sum_{i=1}^d c_i S^i_1+
f(S_1)$ for any constants $c_0,c_1,...,c_d\in\mathbb R$. Hence, Theorems \ref{thm.1}, \ref{thm.2}
hold true for the payoff 
$$f(S_1):=\sum_{i=1}^d a_i S^i_1-K+\left(K-\sum_{i=1}^d a_i |S^i_1|\right)^{+}, \ \ K,a_1,...,a_d\geq 0.$$
For the case where $S^i_1\geq 0$ for all $i=1,...,d$, the above right hand side is equal to the payoff of the call option 
with the payoff 
$\left(\sum_{i=1}^d a_i S^i_1-K\right)^{+}$.

Finally, we notice that barrier type of payoffs which have the form $f(S_1)=g(S_1)\mathbbm{1}_{|S_1|<K}$
where $K>0$ and $g:\mathbb R^d\rightarrow\mathbb R_{+}$ is a bounded lower semi-continuous function 
($\mathbbm{1}$ denotes the indicator function) also satisfy 
Assumption \ref{asm1}. 

\end{rem}

{
\begin{rem}
  It may be possible to establish the existence of classical solutions to \eqref{cauchy} via the continuity method, following the approach of \cite{DaSa12}, and to obtain third-order derivative estimates as in \cite{CNSI}, under weaker assumptions on $h$. In particular, one might be able to prove Proposition~\ref{2.2} under the conditions $h \in C^2(\mathbb R^d)$ and $\Lambda I_d \geq \nabla^2 h \geq \lambda I_d$ for some constants $\Lambda > \lambda>0$. If this can be achieved, the uniform boundedness condition in Assumption~\ref{asm1} could be removed, thereby allowing payoffs such as call options. We leave this for future investigation.
\end{rem}
}

\begin{rem}
Our proof of the lower bound (given in Section \ref{sec:4}) uses only the fact that $f$ is lower semi-continuous and has linear growth.
Let us 
notice that we can not drop the assumption that $f$ is lower semi-continuous. Indeed, take $d=1$, $S_0=\frac{1}{2}$, $b=0$
and consider 
the payoff function 
$$
f(x)=\begin{cases}
			2, & \text{if $x=0$ },\\  
            0, & \text{otherwise}.
		 \end{cases}
$$
Clearly, $f(S_1)=0$, $\mathbb P$ a.s., and so $c_n=0$ for all $n$. 
On the other hand, by applying the main result in \cite{BB} it follows 
that there exists $\Sigma \in\mathcal V$ such that 
$\frac{1}{2}+\int_{0}^1 \sqrt\Sigma_t  \, dW_t\sim Bernoulli(1/2)$
and $\mathbb E_{\mathbb P}\left[\int_{0}^1 G(\Sigma_t)dt\right]=\log(\pi)-\frac{3}{8}$. Hence, 
the right hand side of (\ref{2.3}) is no less than
$1-\log(\pi)+\frac{3}{8}>0$.
\end{rem}

\section{Proof of Theorem \ref{thm.2}}\label{sec:3}
{ We first prove the upper bound by constructing asymptotically optimal strategies as it is the main contribution of the paper. Different from the tightness argument in \cite{CD}
(which was based on artificial assumptions about the financial market),
we adopt a PDE approach. 
In Lemma~\ref{lem3.1} and Lemma~\ref{lem3.2}, we prove that for the right side of \eqref{2.3} replacing the payoff $f$ by its concave envelop does not change the optimal value. In Lemma~\ref{lem.den}, we show that under Assumption~\ref{asm1}, the concave envelop of $f$ can be approximated by a sequence of functions in $\mathcal{H}$. Then in Proposition~\ref{prop2.1}, we argue that taking any element of $\mathcal{H}$ as terminal condition, the PDE \eqref{cauchy} has a classical solution. With all these results, we prove Theorem~\ref{thm.2} at the end of this section. }

For any $K>0$, denote by $\mathcal{S}_d^K$ the set of all positive semi-definite matrices with eigenvalues in the interval $[1/K,K]$, and let $\mathcal V_K\subset \mathcal V$ be the set of all processes $\Sigma$ with values in $\mathcal{S}_d^K$, $ \mathbb P \otimes dt$ a.s.
We start the proof with the following lemma.
\begin{lem}\label{lem3.1}
Let $\phi:\mathbb R^d\rightarrow\mathbb R$ be a concave function with at most quadratic growth. 
Then for any $t\in [0,1)$ and $a>0$, the function $\psi$ given by
$$\psi(x):=\sup_{\Sigma \in\mathcal V_K}\mathbb E_{\mathbb P}\left[\phi\left(x+\int_{t}^1 \sqrt{\Sigma}_udW_u\right)+
\frac{1}{2}\left(\int_{t}^1 \log(\det(\Sigma_u))-a\Tr (\Sigma_u)\right)du \right]$$
is concave.
 \end{lem}
\begin{proof}
From the quadratic growth of $\phi$ it follows that $\psi$ is well-defined.
Fix $t\in [0,1)$, $x_1,x_2\in \mathbb R^d$ and $\lambda\in (0,1)$. Choose $\epsilon>0$. For any $i=1,2$
there exists $\Sigma^i\in\mathcal V_K$
such that 
\begin{equation*}
\psi(x_i)<\epsilon+\mathbb E_{\mathbb P}\left[\phi\left(x_i+\int_{t}^1 \sqrt{\Sigma^i_u}dW_u\right)+\frac{1}{2}\int_{t}^1 \left(\log(\det(\Sigma^i_u))-a\Tr (\Sigma^i_u) \right)du \right].
\end{equation*}
Define $\Sigma:=\left(\lambda\sqrt{\Sigma^1}+(1-\lambda)\sqrt{\Sigma^{2}}\right)^2$, which clearly an element of $\mathcal V_K$.
From the concavity of $\phi$, the concavity of the map $\mathcal{S}_d \ni A\rightarrow \log(\det(A))$, the convexity of the map $\mathcal{S}_d \ni \sigma \mapsto \Tr(\sigma^2)$ ($\sigma=\sqrt\Sigma$),
we obtain that 
\begin{align*}
&\psi(\lambda x_1+(1-\lambda)x_2)\\
&\geq \mathbb E_{\mathbb P}
\left[\phi\left(\lambda x_1+(1-\lambda)x_2+\int_{t}^1 \sqrt\Sigma_u dW_u\right)+\frac{1}{2}\int_{t}^1 \left(\log(\det(\Sigma_u))-a \Tr(\Sigma_u) \right)du \right]\\
&\geq \lambda \mathbb E_{\mathbb P}
\left[\phi\left(x_1+\int_{t}^1 \sqrt{\Sigma^1_u} dW_u\right)+\frac{1}{2}\int_{t}^1 \left(\log(\det(\Sigma^1_u))-a \Tr(\Sigma^1_u)\right)du \right]\\
&+(1-\lambda)\mathbb E_{\mathbb P}
\left[\phi\left(x_2+\int_{t}^1 \sqrt{\Sigma^2_u}dW_u\right)+\frac{1}{2}\int_{t}^1\left(\log(\det(\Sigma^2_u))-a \Tr(\Sigma^2_u)\right)du \right]\\
&\geq\lambda \psi(x_1)+(1-\lambda)\psi(x_2)-\epsilon.
\end{align*}
Letting $\epsilon \to 0$,  we complete the proof. 
\end{proof}
The next lemma implies that in Lemma~\ref{lem3.1}, it is equivalent to  replace the terminal condition $\phi$ by its concave envelop. 
\begin{lem}\label{lem3.2}
 Let $\phi:\mathbb R^d\rightarrow\mathbb R$ be a 
 uniformly bounded
and lower semi-continuous
  function.
 Then for any $x\in\mathbb R^d$
 \begin{align}\label{equal}
  &\sup_{\Sigma \in\mathcal V}\mathbb E_{\mathbb P}\left[\phi\left(x+\int_{0}^1 \sqrt\Sigma_udW_u\right)- \int_{0}^1 G\left(\Sigma_t\right)dt \right]\nonumber\\
  &= \sup_{\Sigma\in\mathcal V}\mathbb E_{\mathbb P}\left[ \hat \phi\left(x+\int_{0}^1\sqrt\Sigma_udW_u\right)- \int_{0}^1 G\left(\Sigma_t\right)dt \right]
 \end{align}
  where $\hat \phi(x):=\frac{1}{2}|x|^2+(\phi-|\cdot|^2/2)^{cav}(x)$.
  \end{lem}
  \begin{proof}
  From the It\^{o} formula it follows that 
  for any function $\xi: \R^d \to \R$ we have 
   \begin{align*}
  &\sup_{\Sigma\in\mathcal V}\mathbb E_{\mathbb P}\left[\xi\left(x+\int_{0}^1\sqrt\Sigma_udW_u\right)- \int_{0}^1 G\left(\Sigma_t\right)dt \right]\\
  &= \frac{1}{2}(|x|^2+d)+\sup_{\Sigma\in\mathcal V}\mathbb E_{\mathbb P}\left[\tilde\xi\left(x+\int_{0}^1 \sqrt\Sigma_udW_u\right)+\frac{1}{2}\int_{0}^1 \log(\det\Sigma_t) dt \right]
 \end{align*}
  where $\tilde\xi(x):=\xi(x)-\frac{1}{2}|x|^2$. Thus, in order to prove (\ref{equal}) we need to establish the equality 
  \begin{align*}
  &\sup_{\Sigma\in\mathcal V}\mathbb E_{\mathbb P}\left[\tilde\phi \left(x+\int_{0}^1\sqrt\Sigma_udW_u\right)+\frac{1}{2}\int_{0}^1 \log(\det(\Sigma_s)) ds \right]\\
  &= \sup_{\Sigma\in\mathcal V}\mathbb E_{\mathbb P}\left[\tilde\phi^{cav}\left(x+\int_{0}^1\sqrt\Sigma_udW_u\right)+\frac{1}{2}\int_{0}^1 \log(\det(\Sigma_s)) ds \right],
 \end{align*}
 where $\tilde \phi^{cav}$ denotes the concave envelop of $\tilde \phi$. 
By applying the dynamical programming principle (see Section 4 in \cite{P}) and the Fatou lemma, the above equality will follow from 
 establishing the statement 
 $$\liminf_{t \uparrow 1} \sup_{\Sigma\in\mathcal V}\mathbb E_{\mathbb P}\left[\tilde\phi\left(x+\int_{t}^1 \sqrt\Sigma_udW_u\right)+
\frac{1}{2}\int_{t}^1 \log(\det(\Sigma_u))du \right]\geq \tilde\phi^{cav}(x), \ \forall \, x.$$
 
To this end, choose $x\in\mathbb R^d$ and fix $\epsilon>0$.
 From basic properties of concave envelop (see Section 2 in \cite{RW}) 
it follows that there exists $v_1,...v_{d+1}\in\mathbb R^d$ and $\lambda_1,...,\lambda_{d+1}\geq 0$
such that $\sum_{i=1}^{d+1}\lambda_i=1$, 
$\sum_{i=1}^{d+1}\lambda_i v_i=x$
and 
$\sum_{i=1}^{d+1} \lambda_i\tilde\phi(v_i)>\tilde\phi^{cav}(x)-\epsilon.$
{Clearly, there exists a measurable function $\iota:[0,1]\times\mathbb R^d\rightarrow \mathbb R^d$ such that 
$\mathbb P\left(\iota(t,W_1-W_t)=v_i\right)=\lambda_i$ for all $t\in [0,1)$ and $i=1,...,d+1$.
From the martingale representation theorem it follows that 
for any $t\in [0,1)$, there exists $\Sigma^t\in\mathcal V$ such that 
$$\iota(t,W_1-W_t)=\mathbb E_{\mathbb P}\left[\iota(t,W_1-W_t)|\mathcal F_t\right]+\int_{t}^1\sqrt{\Sigma^t_u}dW_u=x+\int_{t}^1\sqrt{\Sigma^t_u}dW_u.$$

For any $t$, let $\hat\Sigma^{t}\in\mathcal V$ be given by 
${\hat\Sigma}^t:=\Sigma^t+(1-t) I_d$.  Observe that as $t \to 1$, $$\mathbb E \left[ \left| \left(x+\int_{t}^1\sqrt{\Sigma^t_u}dW_u\right)-\left( x+\int_{t}^1\sqrt{\hat{\Sigma}^t_u}dW_u\right)\right|^2 \right] \to 0,$$
which implies that $x+\int_{t}^1\sqrt{\hat{\Sigma}^t_u}dW_u$, $t\in [0,1)$ converges in distribution to
 a random vector which takes on the values $v_1,...,v_{d+1}$ with the probabilities 
 $\lambda_1,...,\lambda_{d+1}$.} Hence, 
from the lower semi-continuity and the quadratic growth of $\tilde\phi$
we obtain 
$$ \liminf_{t \uparrow 1 }\mathbb E_{\mathbb P}\left[\tilde\phi\left(x+\int_{t}^1\sqrt{\hat{\Sigma}^t_u}dW_u\right)\right] \geq 
\sum_{i=1}^{d+1} \lambda_i\tilde\phi(v_i)>
\tilde \phi^{cav}(x)-\epsilon.$$ 
This together with fact that the 
 eigenvalues of
 $\hat\Sigma^t$ are no smaller than $(1-t)$ gives 
\begin{align*}
& \liminf_{t\uparrow 1 }\mathbb E_{\mathbb P}\left[\tilde\phi\left(x+\int_{t}^1\sqrt{\hat{\Sigma}^t_u}dW_u\right)+
\frac{1}{2}\int_{t}^1 \log(\det(\hat{\Sigma}^t_u))du \right]\\
&\geq \tilde\phi^{cav}(x)-\epsilon+ \frac{1}{2}\liminf_{t\uparrow 1 } d (1-t)\log(1-t) \geq  \tilde\phi^{cav}(x)-\epsilon.
 \end{align*}
Letting $\epsilon \to 0$, we complete the proof. 
 \end{proof}

 Next, we establish the following density result. 
\begin{lem}\label{lem.den}
For any $\epsilon>0$ and a function $f:\mathbb R^d \to \mathbb R$ satisfying Assumption~\ref{asm1}, there exists $h\in \mathcal{H}$ such that \eqref{density} holds true.
\end{lem}
\begin{proof}
Recall the set $B^c_M$. 
From the fact that $f:\mathbb R^d\rightarrow\mathbb R$ and 
$\nabla f:B^c_M\rightarrow\mathbb R^d$ are uniformly bounded
we get that for $y\in B^c_M$ and $z\in\mathbb R^d$
the term 
$|f(y)-f(z)|+\left|\left\langle \nabla f(y), z-y\right\rangle\right|$ is of order $O(1+|z-y|)$ { which is dominated by $|y-z|^2$ for large enough $|y-z|$}. Hence, 
 there exists $\hat M>0$ such that  
for any $y\in B^c_M$ and $z\in\mathbb R^d$ with $|z-y|>\hat M$, we have
\begin{equation}\label{8.1}
|y-z|^2\geq |f(y)-f(z)|+\left|\left\langle \nabla f(y), z-y\right\rangle\right|.
\end{equation}
Let us argue that on the set $B^c_{M+\hat M}:=\{x \in \mathbb R^d: \, |x| >M+\hat M \}$ we have 
$f- | \cdot |^2/2 \equiv\left(f- | \cdot |^2/2 \right)^{cav}$. 

To this end we prove that for any $y\in B^c_{M+\hat M}$ and $z\in\mathbb R^d$, 
\begin{equation}\label{8.2}
f(y)-\frac{1}{2}|y|^2+\left\langle \nabla f(y)-y,z-y\right\rangle\geq f(z)-\frac{1}{2}|z|^2,
\end{equation}
which implies that $$f(y)-\frac{1}{2}|y|^2+\left\langle \nabla f(y)-y, z-y\right\rangle \geq \left(f- | \cdot |^2/2 \right)^{cav}(z ),  \quad  \forall \, z \in \R^d,$$
and hence $ f(y)- |y|^2/2 =\left(f- | \cdot |^2/2 \right)^{cav}(y)$. 
Indeed, if $|z-y|>\hat M$, \eqref{8.2} follows from \eqref{8.1}. 
{ If $|z-y|\leq \hat M$, then for any $\lambda\in (0,1)$ we have $\lambda y+(1-\lambda) z\in B^c_M$. Recall $\nabla^ 2 f(x)\leq (1-\alpha) I_d$ for all $x\in B^c_M$ from Assumption~\ref{asm1}. Therefore, the function $\lambda \mapsto f(\lambda y+(1-\lambda) z)-\frac{1}{2}|\lambda y+(1-\lambda) z|^2$ is concave, which yields \eqref{8.2}. }

Next, define the function $\hat h(x):=\left(f- | \cdot |^2/2 \right)^{cav}(x)+\frac{1}{2} x^2$, $x\in\mathbb R^d$ and for 
any $\delta>0$ set $h_{\delta}(x):=\mathbb E\left[ \hat h(x+\delta Z)\right]$, $x\in\mathbb R^d$ where 
$Z$ is a normal random vector with mean zero and covariance matrix $I_d$. 

From Assumption \ref{asm1} we have that $f_{|B^c_M}$ 
is uniformly bounded and uniformly continuous. Since $B_{M+\hat M}$ is compact then 
from the fact that $\hat h\equiv f$ on the set $B^c_{M+\hat M}$ 
we conclude that $\hat h$ is uniformly bounded and uniformly continuous. 
Hence, 
we have the uniform convergence 
$\lim_{\delta\rightarrow 0}\sup_{x\in\mathbb R^d}|h_{\delta}(x)-\hat h(x)|=0$.

Thus, in order to complete the proof it remains to show that for any $\delta>0$, $h_{\delta}\in\mathcal H$. From the fact that 
$\hat h$ is uniformly bounded we have
$h_{\delta}\in C^{\infty}(\mathbb R^d;\mathbb R)$ and 
the corresponding derivatives of all order are uniformly bounded functions.
Finally, 
observe that 
$h_{\delta}(x)-\frac{1}{2}x^2=\mathbb E\left[ \left(f- | \cdot |^2/2 \right)^{cav}(x+\delta Z)\right]+\frac{1}{2}\delta^2$
and so from  the relations 
$\nabla^ 2 f(x)\leq (1-\alpha) I_d$ for all $x\in B^c_M$ (Assumption \ref{asm1}) and the equality 
$f- | \cdot |^2/2 \equiv\left(f- | \cdot |^2/2 \right)^{cav}$
on $B^c_{M+\hat M}$ we conclude that (recall that $B_{M+\hat M}$ is compact) there exists 
$\epsilon:=\epsilon(\delta)$ such that { 
\begin{align*}
\nabla^2 \left(\frac{1}{2}x^2- h_{\delta}(x)\right) \geq &- \mathbb E\left[ \nabla^2 \left(f- | \cdot |^2/2 \right)^{cav}(x+\delta Z) \mathbbm{1}_{\{(x+\delta Z) \in B_M^c \}}\right]  \\
\geq & \alpha \mathbb P[(x+\delta Z) \in B_M^c ] I_d  \geq \epsilon I_d
\end{align*}
} for all $x\in\mathbb R^d$
and the proof is completed.
\end{proof}

The following result is essential to our analysis. We prove that the concavity of the terminal condition passes to that of the value function.  
\begin{prop}\label{prop2.1}
Let $h\in \mathcal H$. 
For $(t,x)\in [0,1]\times\mathbb R^d$, set 
\begin{equation}\label{3.1+}
u(t,x):=\sup_{\Sigma\in\mathcal V}\mathbb E_{\mathbb P}\left[h\left(x+\int_{t}^{1}\sqrt{\Sigma_u}dW_u\right)-\int_{t}^{1} G(\Sigma_u) du\right].
\end{equation}
Then, 
$u\in C^{1,2}([0,1]\times\mathbb R^d)$ 
is the unique solution to the Cauchy problem (\ref{cauchy}), and $\partial_t u$ and $\nabla^2 u $ are Lipschitz continuous. 
Moreover, there exists a constant $\alpha>0$ such that (as usual we use the Loewner order)
\begin{equation}\label{3.2}
\nabla^2u\leq (1-\alpha) I_d, \ \ \forall (t,x)\in [0,1]\times\mathbb R^d.
\end{equation} 
\end{prop}
\begin{proof}
For any $K>0$ introduce the function,
$$
u^K(t,x):=\sup_{\Sigma\in\mathcal V_K}\mathbb E_{\mathbb P}\left[ h\left(x+\int_{t}^{1}\sqrt{\Sigma_u}dW_u\right)-\int_{t}^{1} G(\Sigma_u) du\right], \ \ 
(t,x)\in [0,1]\times\mathbb R^d.
$$
It is a standard result (see e.g. {\cite[Chapter 4, Theorem 4.3.1, Theorem 4.4.3]{P}}) that $u^K$ is the unique viscosity solution to the PDE, 
 \begin{equation}\label{3.1000}
 \begin{cases}
  - \partial_t u^K(t,x)=\sup\limits_{B \in\mathcal S_d^K} \left\{-G(B)+\frac{1}{2}\Tr (B\nabla^2 u^K(t,x))\right\},  \\
 \ \ \  \, \  u^K(1,x)=h(x).
 \end{cases}
 \end{equation}

From the fact that $h\in \mathcal H$, due to the uniform ellipticity of \eqref{3.1000},
 by {\cite[Section 5.5, Theorem 2]{K}} there exists a classical solution $v$ to \eqref{3.1000} such that $\partial_t v$ and $\nabla^2 v$ are Lipschitz continuous. By uniqueness, we know that $v=u^K$. In the rest, we prove that for large enough $K$, $u^K$ is also a solution to \eqref{cauchy}.

Since $h\in\mathcal H$, there exists a constant $C>0$ such that 
$h_1(x):=h(x)+C|x|^2$ is a convex function and 
$h_2(x):=h(x)+\left(\frac{1}{C}-\frac{1}{2}\right)|x|^2$ is a concave function.
Let $t\in [0,1]$. From the It\^{o} isometry we obtain 
\begin{align*}
&u^K(t,x)+C|x|^2\\
&=\sup_{\Sigma\in\mathcal V_K}\mathbb E_{\mathbb P}\left[h_1\left(x+\int_{t}^{1}\sqrt{\Sigma_u}dW_u\right)-\int_{t}^{1} \left(G(\Sigma_u)+
C \Tr\left(\Sigma_u \right)\right)du\right].
\end{align*}
Since for each fixed $\Sigma \in \mathcal{V}_K$, 
$$x \mapsto \mathbb E_{\mathbb P}\left[h_1\left(x+\int_{t}^{1}\sqrt{\Sigma_u}dW_u\right)-\int_{t}^{1} \left(G(\Sigma_u)+C \Tr\left(\Sigma_u \right)\right)du\right]$$
is convex, as the supremum of convex functions,
$x\rightarrow u^K(t,x)+ C |x|^2$ is convex as well, and thus 
\begin{align}\label{eq:c1}
   I_d- \nabla^2 u^K(t,x) \leq (1+ 2C)I_d, \quad (t,x) \in [0,1] \times \R^d.
\end{align}
Again, by applying the It\^{o} isometry it follows
\begin{align*}
&u^K(t,x)+\frac{1}{C}|x|^2=\frac{1}{2}\left(|x^2|+d(1-t)\right)\\
&+\sup_{\Sigma\in\mathcal V_K}\mathbb E_{\mathbb P}\left[ h_2\left(x+\int_{t}^{1}\sqrt{\Sigma_u}dW_u\right)+\int_{t}^{1} \left(\frac{1}{2}\log(\det(\Sigma_u))-\frac{\Tr(\Sigma_u)}{C}\right)du\right].
\end{align*}
From Lemma \ref{lem3.1} it follows that 
 for any $t\in [0,1]$,
$x\rightarrow u^K(t,x)+\left(\frac{1}{C}-\frac{1}{2}\right)|x|^2$ is a concave function, and hence 
\begin{align}\label{eq:c2}
I_d-\nabla^2 u^K (t,x) \geq \frac{2}{C} I_d, \quad (t,x) \in [0,1] \times \R^d .
\end{align}

Observe that the constant $C$ only depends on $h$ but not on $K$. For a given $A\in\mathcal S_d$, the function $\mathcal{S}_d \ni B \rightarrow \log(\det(B))-\Tr(B)+\Tr(BA)$ is strictly concave.  Hence, the first order condition implies that the maxima of 
\begin{align*}
  &\sup_{B \in \mathcal{S}_d}   \left\{-G(B)+\frac{1}{2}\Tr(B\nabla^2 u^K) \right\} \\
  &=\frac{1}{2}\sup_{B \in \mathcal{S}_d}   \left\{\log(\det(B))+ d -\Tr(B)+\Tr(B\nabla^2 u^K) \right\},
\end{align*}
is obtained at $B^*=\left(I_d -\nabla^2 u^K \right)^{-1}$, and hence 
$$\sup_{B \in \mathcal{S}_d}   \left\{-G(B)+\frac{1}{2}\Tr(B\nabla^2 u^K) \right\}=-\frac{1}{2}\log\left(\det\left(I_d-\nabla^2 u^K\right)\right).$$
Furthermore, due to \eqref{eq:c1} and \eqref{eq:c2}, we have that $\frac{1}{1+2C} I_d \leq B^* \leq \frac{C}{2}I_d.$
Therefore, for any $K \geq 1+2 C$
\begin{align*}
&=\sup_{B\in\mathcal S^K_d} \left\{-G(B)+\frac{1}{2}\Tr (B\nabla^2 u^K)\right\}\\
&=
\sup_{B\in\mathcal S_d} \left\{-G(B)+\frac{1}{2}\Tr (B\nabla^2u^K)\right\}=
-\frac{1}{2}\log\left(\det\left(I_d - \nabla^2 u^K\right)\right).
\end{align*}
Namely, for sufficiently large $K$, $u^K$ solves the Cauchy problem (\ref{cauchy}) and hence  $u_K=u$ by the uniqueness. Finally, \eqref{3.2} follows from \eqref{eq:c2}. 
\end{proof}

Now, we have all the ingredients for the proof of Theorem \ref{thm.2}.
\begin{proof}
In view of Lemma \ref{lem.den}, let $\epsilon>0$ and let $h\in \mathcal{H}$ satisfy (\ref{density}). Take $u$ as defined in \eqref{3.1+} with the terminal condition $h$. From Lemma \ref{lem3.2} and Proposition \ref{prop2.1} it follows that 
\begin{equation}\label{4.2}
u(0,S_0)\leq \epsilon+ \sup_{\Sigma\in\mathcal V}\mathbb E_{\mathbb P}\left[f\left(S_0+\int_{0}^1\sqrt{\Sigma_u}dW_u\right)- \int_{0}^1 G\left(\Sigma_t\right)dt \right].
\end{equation}
Fix $n\in\mathbb N$. 
For $j=0,1,....,n-1$, define the random variables
$$I^n_j:=u\left(\frac{j+1}{n},S_{\frac{j+1}{n}}\right)-u\left(\frac{j}{n},S_{\frac{j}{n}}\right)
-\left\langle \nabla u\left(\frac{j+1}{n},S_{\frac{j}{n}}\right),S_{\frac{j+1}{n}}-S_{\frac{j}{n}}\right\rangle .$$
From (\ref{density}) we have  
\begin{align*}
& \frac{1}{n}\log\left(\mathbb E\left[e^{n\left(f(S_1)-\sum_{i=0}^{n-1} \left\langle \nabla u\left(\frac{i+1}{n},S_{\frac{i}{n}}\right),S_{\frac{i+1}{n}}-S_{\frac{i}{n}}\right\rangle\right)}\right]\right)\nonumber\\
&\leq\epsilon+\frac{1}{n}\log\left(\mathbb E\left[e^{n\left(u(1,S_1)-\sum_{i=0}^{n-1} \left\langle \nabla u \left(\frac{i+1}{n},S_{\frac{i}{n}}\right),S_{\frac{i+1}{n}}-S_{\frac{i}{n}}\right\rangle\right)}\right]\right)\nonumber\\
&=\epsilon+u(0,S_0)+\frac{1}{n}\log\left(\mathbb E\left[e^{n\sum_{j=0}^{n-1} I^n_j}\right]\right).\\
\end{align*}
{In view of \eqref{4.2}, in order to complete the proof of Theorem \ref{thm.2}, it is sufficient 
prove that there exists a constant $L>0$ such that 
for any $n$ and $j=0,1,...,n-1$, the inequality holds
\begin{equation}\label{4.4}
\mathbb E\left[e^{n I^n_j}\left|\right.\mathcal F_{\frac{j}{n}}\right]\leq e^{\frac{L}{\log n}}.
\end{equation}}
Indeed  with \eqref{4.4}, by tower property of conditional expectation 
\begin{align*}
\mathbb E \left[e^{n \sum_{j=0}^{n-1}I_j^n } \right]= &\mathbb E \left[e^{n \sum_{j=0}^{n-2}I_j^n } \, \mathbb E \left[ e^{nI_{n-1}^n}   \big| \, \mathcal{F}_{\frac{n-1}{n}}\right] \right] 
\leq  e^{\frac{L}{\log n}} \mathbb E\left[e^{n \sum_{j=0}^{n-2}I_j^n } \right],
\end{align*}
and hence with the same argument $\mathbb E \left[e^{n \sum_{j=0}^{n-1}I_j^n } \right] \leq e^{\frac{L n}{\log n}}$. It implies that $$ \lim\limits_{n \to \infty} \frac{1}{n}\log\left(\mathbb E\left[e^{n\sum_{j=0}^{n-1} I^n_j}\right]\right) \leq \lim\limits_{n \to \infty} \frac{L}{\log n} =0. $$

From the Taylor formula, \eqref{3.2}, and the Lipschitz property of $\partial_t u, \nabla^2 u$,
we obtain that for sufficiently large $n$
\begin{align}\label{4.5}
I^n_j=&u\left(\frac{j+1}{n},S_{\frac{j}{n}}\right)-u\left(\frac{j}{n},S_{\frac{j}{n}}\right)\nonumber\\
&+u\left(\frac{j+1}{n},S_{\frac{j+1}{n}}\right)-u\left(\frac{j+1}{n},S_{\frac{j}{n}}\right)\nonumber\\
&-\left\langle \nabla u\left(\frac{j+1}{n},S_{\frac{j}{n}}\right),S_{\frac{j+1}{n}}-S_{\frac{j}{n}}\right\rangle \nonumber\\
\leq & \frac{1}{n\log n}+\frac{1}{n} \partial_t u\left(\frac{j+1}{n},S_{\frac{j}{n}}\right)+K^n_j,
\end{align}
where $K^n_j$ is a second-order estimate of in the  Taylor expansion
\begin{align*}
K^n_j:=&\mathbbm{1}_{\left|S_{\frac{j+1}{n}}-S_{\frac{j}{n}}\right| \leq \frac{\log n}{\sqrt n}}
\frac{1}{2}\left(S_{\frac{j+1}{n}}-S_{\frac{j}{n}}\right)^{\top}\nabla^2 u\left(\frac{j+1}{n},S_{\frac{j}{n}}\right)\left(S_{\frac{j+1}{n}}-S_{\frac{j}{n}}\right)\\
&+\mathbbm{1}_{\left|S_{\frac{j+1}{n}}-S_{\frac{j}{n}}\right|>\frac{\log n}{\sqrt n}}\frac{1}{2} (1-\alpha) \left|S_{\frac{j+1}{n}}-S_{\frac{j}{n}}\right|^2.
\end{align*}

Fix $j$. 
Observe that the random vector 
$Z:=\sqrt n\left(S_{\frac{j+1}{n}}-S_{\frac{j}{n}}-\frac{b}{n}\right)$
has a normal distribution with mean zero and covariance matrix $I_d$. 
Thus,
\begin{align*}\label{4.6}
&\mathbb E_{\mathbb P}\left[e^{n K^n_j}\left|\right.\mathcal F_{\frac{j}{n}}\right]\\
&=
\mathbb E_{\mathbb P}\left[e^{\frac{1}{2}\left(Z+\frac{b}{\sqrt n}\right)^{\top}\nabla^2 u\left(\frac{j+1}{n},S_{\frac{j}{n}}\right)\left(Z+\frac{b}{\sqrt n}\right)} \, \Big| \, \mathcal F_{\frac{j}{n}}\right]\nonumber\\
&-\mathbb E_{\mathbb P}\left[\mathbbm{1}_{\left|Z+\frac{b}{\sqrt n}\right|> \log n}e^{\frac{1}{2}\left(Z+\frac{b}{\sqrt n}\right)^{\top} \nabla^2 u\left(\frac{j+1}{n},S_{\frac{j}{n}}\right)\left(Z+\frac{b}{\sqrt n}\right)} \, \Big| \, \mathcal F_{\frac{j}{n}}\right]\nonumber\\
&+\mathbb E_{\mathbb P}\left[\mathbbm{1}_{\left|Z+\frac{b}{\sqrt n}\right|> \log n}e^{\frac{1}{2}(1-\alpha)\left|Z+\frac{b}{\sqrt n}\right|^{2}} \, \Big| \, \mathcal F_{\frac{j}{n}}\right].\nonumber
\end{align*}
Since $Z$ is independent of $\mathcal F_{\frac{j}{n}}$ and $\nabla^2 u\leq (1-\alpha)I_d$, by direct computation there exists a constant 
$L_1>0$ such that the difference 
\begin{align*}
&\left|\mathbb E_{\mathbb P}\left[e^{\frac{1}{2}\left(Z+\frac{b}{\sqrt n}\right)^{\top} \nabla^2 u\left(\frac{j+1}{n},S_{\frac{j}{n}}\right)\left(Z+\frac{b}{\sqrt n}\right)} \, \Big| \, \mathcal{F}_{\frac{j}{n}}\right]  -\sqrt{\det\left(I_d-\nabla^2 u\left(\frac{j+1}{n},S_{\frac{j}{n}}\right)\right)}
\right|
\end{align*}
is bounded from above by $\frac{L_1}{n}$. Furthermore, using the Gaussian distribution of $Z$, by direction computation, there exists an another constant $L_2>0$ such that 
\begin{align*}
&\mathbb E_{\mathbb P}\left[\mathbbm{1}_{\left|Z+\frac{b}{\sqrt n}\right|> \log n}e^{\frac{1}{2}\left(Z+\frac{b}{\sqrt n}\right)^{\top} \nabla^2 u\left(\frac{j+1}{n},S_{\frac{j}{n}}\right)\left(Z+\frac{b}{\sqrt n}\right)} \, \Big| \, \mathcal F_{\frac{j}{n}}\right]\nonumber\\
&+\mathbb E_{\mathbb P}\left[\mathbbm{1}_{\left|Z+\frac{b}{\sqrt n}\right|> \log n}e^{\frac{1}{2}(1-\alpha)\left|Z+\frac{b}{\sqrt n}\right|^{2}} \, \Big| \, \mathcal F_{\frac{j}{n}}\right]\leq \frac{L_2}{\log n}.
\end{align*}
This together with (\ref{4.5}) 
and the relation 
$e^{u_t}=\sqrt{\det\left(I_d-\nabla^2 u\right)}$ gives (\ref{4.4}) and
completes the proof. 
\end{proof}

\section{Proof of Theorem \ref{thm.1}}\label{sec:4}
We start with the following lemma.
\begin{lem}\label{lem4.1}
Consider the measurable (Borel) space
$\hat\Omega:=\mathbb R^d$, $\hat{\mathcal F}:=\mathcal B(\mathbb R^d)$, i.e. the Borel $\sigma$-algebra. 
Let $Y:\hat\Omega\rightarrow\mathbb R^d$ be the corresponding canonical vector. 
For any $\lambda \in\mathbb R^d$ and $\Sigma\in\mathcal S_d$, denote by
$\mathbb P_{\lambda,\Sigma}$ the Gaussian probability measure on 
$(\hat\Omega,\hat{\mathcal F})$ with mean $\lambda$ and 
covariance matrix $\Sigma$. Let $\mathbb E_{\lambda,\Sigma}$ be the 
corresponding expectation. Then, for any measurable function 
$\phi:\mathbb R^d\rightarrow\mathbb R$ and $\lambda\in\mathbb R^d$ we have 
$$
\log\left(\inf_{\gamma\in\mathbb R^d}\mathbb E_{\lambda,I_d}\left[e^{\phi(Y)-\langle \gamma ,Y\rangle }\right]\right)
\geq 
\sup_{\Sigma\in\mathcal S_d}\left(
\mathbb E_{0,\Sigma}\left[\phi(Y)\right]-G(\Sigma)\right)-\frac{1}{2}|\lambda|^2.
$$
\end{lem}
\begin{proof}
We will apply the classical 
duality for exponential hedging in one period model (see Chapter 3 in \cite{F}). In fact
we only need the ``simple" inequality of the duality. The duality is given by 
\begin{equation}\label{4.200}
\log\left(\inf_{\gamma\in\mathbb R^d}\mathbb E_{\lambda,I_d}\left[e^{\phi(Y)-\langle \gamma ,Y\rangle }\right]\right)= \sup_{\mathbb Q\in\mathcal M}
\mathbb E_{\mathbb Q}\left[\phi(Y)-\log\left(\frac{d\mathbb Q}{d{\mathbb P}_{\lambda,I_d}}\right)\right]
\end{equation}
where $\mathcal M$ is the set of all probability measures 
which satisfy $\mathbb E_{\mathbb Q}[Y]=0$
and have a finite relative entropy with respect to $\mathbb P_{\lambda,I_d}$.

Clearly, for any $\Sigma\in\mathcal S_d$ we have $\mathbb E_{0,\Sigma}[Y]=0.$
Furthermore, 
\begin{align*} 
&\mathbb E_{0,\Sigma}\left[\log\left(\frac{d\mathbb P_{0,\Sigma}}{d{\mathbb P}_{\lambda,I_d}}\right)\right]\\
&=\mathbb E_{0,\Sigma}\left[\log\left(\frac{1}{\sqrt{\det(\Sigma)}}\frac{\exp\left(-\frac{1}{2}Y^{\top}\Sigma^{-1}Y\right)}
{\exp\left(-\frac{1}{2}(Y-\lambda)^{\top}(Y-\lambda)\right)}\right)\right]\\
&=\frac{1}{2}\left(|\lambda|^2-\log\left(\det(\Sigma)\right)\right)+\frac{1}{2}\mathbb E_{0,\Sigma}\left[Y^{\top}(I_d-\Sigma^{-1})Y\right]\\
&=\frac{1}{2}|\lambda|^2+G(\Sigma).
\end{align*}
This together with (\ref{4.200}) completes the proof. 
\end{proof}

Now, we are ready to prove Theorem \ref{thm.1}.
\begin{proof}
In view of Theorem \ref{thm.2}, in order to prove Theorem \ref{thm.1} it is sufficient to establish the following lower bound.  
\begin{equation}\label{4.1}
 \liminf_{n\rightarrow\infty} c_n\geq
 \sup_{\Sigma\in\mathcal V}\mathbb E_{\mathbb P}\left[f\left(S_0+\int_{0}^1\sqrt{\Sigma_u}dW_u\right)- \int_{0}^1 G\left(\Sigma_t\right)dt \right].
\end{equation}

For any $K>0$ and $n\in\mathbb N$, let $\mathcal V^n_K\subset\mathcal V$ be the set of all volatility processes
of the form
$$
\Sigma_t=\sum_{k=0}^{n-1} \mathbbm{1}_{t\in \left[\right.\frac{k}{n}, \frac{k+1}{n}\left)\right.}\phi_k\left(W_0,W_{\frac{1}{n}},..., W_{\frac{k}{n}}\right)
$$
where $\phi_k:\mathbb R^{k+1}\rightarrow \mathcal S_d$, $k=0,...,n-1$ are measurable functions
which satisfy $\phi_k\geq \frac{1}{K}I_d$. 
First, we argue that
\begin{align}\label{1}
&\sup_{\Sigma\in\mathcal V}\mathbb E_{\mathbb P}\left[f\left(S_0+\int_{0}^1\sqrt{\Sigma_u}dW_u\right)- \int_{0}^1 G\left(\Sigma_u\right)du \right]\nonumber\\
&=\lim_{K\rightarrow \infty}\liminf_{n\rightarrow\infty}
\sup_{\Sigma\in\mathcal V^n_K}\mathbb E_{\mathbb P}\left[f\left(S_0+\int_{0}^1\sqrt{\Sigma_u}dW_u\right)- \int_{0}^1 G\left(\Sigma_u\right)du \right].
\end{align}

To this end, observe that for any $\Sigma \in \mathcal{S}_d$, $\log\left(\det (\Sigma)\right) \leq \Tr(\Sigma) -d$. Therefore for any $\Sigma \in \mathcal{V}$ by the dominated convergence theorem, 
$$
\lim_{m\rightarrow\infty}\mathbb E_{\mathbb P}\left[\int_{0}^1 G\left(\Sigma_t+\frac{I_d}{m}\right)dt \right]=\mathbb E_{\mathbb P}\left[\int_{0}^1 G\left(\Sigma_t\right)dt \right].
$$
Moreover, from the fact that $f$ is uniformly bounded and lower semi-continuous we obtain 
$$
\mathbb E_{\mathbb P}\left[f\left(S_0+\int_{0}^1\sqrt{\Sigma_u}dW_u\right)\right]\leq \liminf_{m\rightarrow\infty}
\mathbb E_{\mathbb P}\left[f\left(S_0+\int_{0}^1\sqrt{\Sigma_u+\frac{I}{m}}dW_u\right)\right].
$$
Hence, we will get the same supremum on the
right-hand side of (\ref{2.3})
 if instead of letting $\Sigma$ vary over all of $\mathcal V$ there, we confine it to be uniformly bounded away from zero.

Thus, let us assume that
$\Sigma \geq \frac{1}{K} I_d$ for some $K>0$.
From standard density arguments there exists a sequence 
$\Sigma^n\in \mathcal V^n_K$ such that $\Sigma^n \geq \frac{1}{K} I_d$ for all $n$ and $\Sigma^n\rightarrow \Sigma$
in $L^1(\mathbb P\otimes dt)$; see e.g. \cite[Lemma 7.3]{DoSo13}.
Since $M \mapsto G(M)$ is Lipschitz continuous over the set of positive semi-definite matrices $M$ such that $M \geq \frac{1}{K} I_d$, it holds that
$$\mathbb E_{\mathbb P}\left[\int_{0}^1 G\left(\Sigma_t\right)dt \right]=
\lim_{n\rightarrow\infty}\mathbb E_{\mathbb P}\left[\int_{0}^1 G\left(\Sigma^n_t\right)dt \right].$$
Again, from the fact that $f$ is lower semi-continuous with linear growth
$$\mathbb E_{\mathbb P}\left[f\left(S_0+\int_{0}^1\sqrt{\Sigma_u}dW_u\right)\right]\leq \liminf_{n\rightarrow\infty}
\mathbb E_{\mathbb P}\left[f\left(S_0+\int_{0}^1\sqrt{\Sigma^n_u}dW_u\right)\right].$$
We conclude that (\ref{1}) holds true.

Finally, in view of (\ref{1}), in order to get \eqref{4.1}, it is sufficient to show that for any $n\in \mathbb N$
and $K>0$  
\begin{equation}\label{4.400}
c_n\geq \sup_{\Sigma\in\mathcal V^n_K}\mathbb E_{\mathbb P}\left[f\left(S_0+\int_{0}^1\sqrt{\Sigma_u}dW_u\right)- \int_{0}^1 G\left(\Sigma_u\right)du \right]-
\frac{1}{2n}|b|^2.
\end{equation}
Fix $n \in \mathbb N$. From a standard dynamic programming argument, we obtain that  $c_n=V_0(S_0)$,
where the functions $V_0,...,V_n$ are given by the backward recursions
\begin{align*}
V_n(x)=&f(x), \\
V_k(x)=&\frac{1}{n}\inf_{\gamma\in\mathbb R^d} \log\mathbb E_{\mathbb P}
\left[e^{nV_{k+1}\left(x+\left(S_{\frac{k+1}{n}}-S_{\frac{k}{n}}\right)\right)-
\left\langle \gamma ,S_{\frac{k+1}{n}}-S_{\frac{k}{n}}\right\rangle }\right], \ \ k \leq n-1.
\end{align*}
Choose $k<n$ and $x\in\mathbb R^d$. Observe that the random vector
$\sqrt n\left(S_{\frac{k+1}{n}}-S_{\frac{k}{n}}\right)$
has a normal distribution with mean $\frac{b}{\sqrt n}$ and covariance matrix $I_d$.
Thus, by applying Lemma \ref{lem4.1} for the function $\phi(Y):=nV_{k+1}\left(x+\frac{Y}{\sqrt n}\right)$
and $\lambda:=\frac{b}{\sqrt n}$ we obtain that for any $k<n$ and $x\in\mathbb R^d$ 
$$V_k(x)\geq \sup_{\Sigma\in\mathcal S_d}
\left(\mathbb E_{\mathbb P}\left[V_{k+1}\left(x+\sqrt{\Sigma}\left(W_{\frac{k+1}{n}}-W_{\frac{k}{n}}\right)\right)\right]-\frac{1}{n}G\left(\Sigma\right)\right)-\frac{1}{2n^2}|b|^2.$$ 
Hence, from dynamical programming we conclude that 
$V_n(S_0)$ is bigger or equal to the right hand side of (\ref{4.400}).
\end{proof}

\end{document}